\documentclass[final,times]{elsarticle}
\usepackage{commath}
\usepackage[english]{babel}
\usepackage{amsmath,amsthm}
\usepackage{amsfonts}
\usepackage{mathscinet}
\usepackage{textcomp}
\usepackage{units}
\usepackage{enumerate}
\usepackage{srcltx}
\usepackage{graphics}
\usepackage{color}
\usepackage{mathrsfs}
\usepackage{amsthm}
\usepackage{rotating}
\usepackage{tikz}
\usepackage{graphicx}
\usepackage[bookmarks=true, pdfstartview={FitH}]{hyperref}
\usepackage[margin=1in]{geometry}

\usepackage[arrow,matrix,curve]{xy}
\newdir{ >}{{}*!/-1em/\dir{>}}

\newcommand{\lra}{\longrightarrow}

\newcommand{\os}{\overset}

\newcommand{\Hom}{\operatorname{Hom}}
\newcommand{\Ext}{\operatorname{Ext}}

\newenvironment{dedication}
{\vspace{6ex}\begin{quotation}\begin{center}\begin{em}}
			{\par\end{em}\end{center}\end{quotation}}

\theoremstyle{dgthm}
\newtheorem{theorem}{Theorem}
\newtheorem{corollary}{Corollary}
\newtheorem{lemma}{Lemma}

\theoremstyle{dgdef}

\begin{document}
\begin{dedication}
	\hspace{\fill}
\end{dedication}

\title{Strong Shape Invariance of Alexander-Spanier Normal Homology Theory}

\author{Anzor Beridze}

\address{Department of Mathematics,
Faculty of exact sciences and education,
	Batumi Shota Rustaveli State University,
	35, Ninoshvili St., Batumi,
	Georgia; e-mail:~a.beridze@bsu.edu.ge}


\begin{abstract}  In the paper [Ba-Be-Mdz], using the Alexander-Spanier cochains based on the normal coverings, the exact homology theory $\bar{H}^N_*(-,-;G)$, the so called Alexander-Spanier normal homology theory, is defined. In the paper, we use the method of	construction of the strong homology theory to show that the homology theory $\bar{H}^N_*(-,-;G)$ is a strong shape invariant.
\end{abstract}


\maketitle

\section*{\bf Introduction} 
 {\color{black} On the category of pairs of compact metric spaces, an exact homology theory, the so called Steenrod homology theory, was defined  by N. Steenrod \cite{St}. It was extended to the category of pairs of compact spaces and on the subcategory of pairs of compact metric spaces axiomatically described by J.Milnor. Later, there were many approaches to define the exact homology theories  \cite{BM},\cite{Mas1},\cite{Mas2},\cite{Ku}, \cite{Skl}.  However, these approaches gave the unique homology theory on the category pairs of compact Hausdorff spaces \cite{Ku}, \cite{Skl}, \cite{Be},\cite{IM}, \cite{Mdz}, \cite{In}, \cite{BM2}. For general topological spaces,  an exact homology theory was constructed in \cite{Mdz2}. According to Mardesic \cite{Mar}, these groups are to be called 1-height homology groups (\S 17.1, p. 333, \cite{Mar}). What today is called the strong (not finite height)  homology groups, which is strong shape invariant, was defined a bit later \cite{LM1}, \cite{LM2}, \cite{LM3}, \cite{Mim}. However, in the paper \cite{Mdz} these homology groups are called the total homology groups (see bibliographic notes of \S 17 in \cite{Mar}). Note that, in some sense, the finite height homology groups can be considered as an approximation of the strong homology groups \cite{Mar}. In the paper \cite{Mdz2},  these groups are simply called fragments of the total homology groups. The essential formulas, which give the connection of strong (total) homology  and finite height (fragment) homology groups are the sequences:
 \[
 \begin{tikzpicture}
 	\node (A) {$\dots$};
 	\node (B) [node distance=1.4cm, right of=A] {$  \bar{H}^{(r-2)} _{n+1}(\mathbf{C} )$};
 	\node (C) [node distance=2.2cm, right of=B] {${\varprojlim} ^ {r}  H_{n+r}(\mathbf{C})$};
 	\node (D) [node distance=4.2cm, right of=B] {$ \bar{H}^{(r)}_n(\mathbf{C})$};
 	\node (E) [node distance=1.9cm, right of=D] {$ \bar{H}^{(r-1)}_n(\mathbf{C})$};
 	\node (F) [node distance=2.4cm, right of=E] {${\varprojlim}^{r+1}  {H} _{n+r}(\mathbf{C}) $};
 	\node (H) [node distance=1.9cm, right of=F] {$\dots$,};

 	\draw[->] (A) to node [above]{}(B);
 	\draw[->] (B) to node [above]{}(C);
 	\draw[->] (C) to node [above]{}(D);
 	\draw[->] (D) to node [above]{}(E);
 	\draw[->] (E) to node [above]{}(F);
 	\draw[->] (F) to node [above]{}(H);
 	
 \end{tikzpicture}
 \]
 \[\footnotesize
 0 \to {\varprojlim}^1 \bar{H}^{(*)}_{n+1}(\mathbf{C}) \xrightarrow{\text{{  }}} \bar{H}_n(\mathbf{C}) \xrightarrow{\text{}} \varprojlim \bar{H}^{(*)}_n(\mathbf{C}) \to 0.
 \]
 The announcement about these formulas was in  \cite{Mim}, but  the proofs for the first time appeared in  \cite{Mdz}, \cite{Mdz2}, \cite{MM}, \cite{MP2} (see bibliographic notes of \S 17 in \cite{Mar}).	Using these formulas, for the strong homology theory the Milnor's short sequence \cite[Theorem 2]{MP2} and  the universal coefficient formulas \cite[Theorem 4]{MP2} are obtained  on the category $\mathcal{K}^2_{C}$ of compact pairs. In particular:
 \begin{itemize}
 	\item it is shown that for every compact Hausdorff space $X$, abelian group $G$ and integer $n \in \mathbb{N}$, there exists an exact sequence
 	\[\footnotesize
 	0 \to {\varprojlim}^1 {H}_{n+1}(\mathbf{X};G) \xrightarrow{\text{{  }}} \bar{H}_n(X;G) \xrightarrow{\text{}} \varprojlim {H}_n(\mathbf{X};G) \to 0.
 	\]
 	which is natural in both variables. Here $\mathbf{X}=\{X_\alpha, p_{\alpha \alpha'}, \Lambda\}$  is an inverse system of compact polyhedra with $\varprojlim \mathbf{X}=X$ \cite{MP2}. Moreover, it is shown that there exists a paracompact space $X$ such that $\check{H}(X;G)= \varprojlim {H}_n(\mathbf{X};Z)=0$ and  ${H}_{n+1}(\mathbf{X};Z)=0$, but $\bar{H}_n(X;G) \neq 0$ \cite[Remark 3]{MP2}. Therefore, the formula is not valid for paracompact spaces.
 	\item For every  compact Hausdorff space $X$, abelian group $G$ and integer $n \in \mathbb{N}$, there exists an exact sequence
 	\[0 \to Ext(\check{H}^{n+1}(X);G) \to \bar{H}_n(X;G) \to Hom(\check{H}^n(X);G) \to 0,\]
 	which is natural in both variables, where $\check{H}^{n+1}(X)$ is Alexandrov-\v{C}ech cohomology \cite{MP2}. An example is given, which shows that for the strong homology theory  $\bar{H}_{*}(-,-;G)$, the Universal Coefficient Formula is false even for polyhedrons \cite[Remark 6]{MP2}. Indeed, if $X=\vee _i X_i,$ where $X_i=P^2$ are projective planes and $G=\mathbb{Z}_2,$ then
 	\[\check{H}^3(X;G)=0,~~\bar{H}_2(X;G)=\oplus_i\mathbb{Z}_2,~~\check{H}^2(X;G)=\prod_i \mathbb{Z}_2 .\]
 	Therefore, since $\oplus_i\mathbb{Z}_2 \not \simeq Hom \left(\prod_i \mathbb{Z}_2 , \mathbb{Z}_2 \right)$ the universal coefficient formula is not fulfilled.  
 \end{itemize}
Therefore, by the uniqueness theorems obtained by Berikashvili, and as well as Inasaridze and Mdzinarishvili, the strong homology theory is  isomorphic to the Steenrod homology theory for pairs of compact topological spaces \cite{Be}, \cite{IM}, \cite{BM2}.

Recently, another exact homology theory  on the category $\mathcal{K}^2_{Top}$ of closed P-pairs \cite{MS} of general topological spaces, the so called the Alexander-Spanier normal homology theory $\bar{H}^N_*(-,-;G)$, has been constructed  by the paper \cite{BBM}, which satisfies the Universal Coefficient Formula for any topological spaces \cite[Corrolary 9]{BBM}. Therefore, by the paper \cite{Be},  the Alexander-Spanier normal homology theory and the strong homology theories are isomorphic on the subcategory of pairs of compact topological spaces. On the other hand, by the Uniqueness Theorem \cite[Theorem 7]{BBM}, \cite[Corollary 10]{BBM}, since for the strong homology theory the Universal Coefficient Formula is false in general, we can conclude that  the Alexander-Spanier normal homology theory and the strong homology theory are different for non-compact spaces. Therefore, it is interesting to investigate properties of the homology theory $\bar{H}^N_*(-,-;G)$. Partially, it is done in the paper \cite{BBM}. In particular, it is shown that the Alexander-Spanier normal homology theory satisfies the Eilenberg-Steenrod axioms except of dimension axiom in the general case. Moreover, it is shown that the following properties are fulfilled:}

  $CEH$ ({\it Continuity for an Exact Homology}): For each resolution  ${\bf p}:(X,A) \to {\bf (X,A)}=\left\{(X_\lambda, A_\lambda), p_{\lambda \lambda '}, \Lambda \right\} $ of closed $P$-pair $(X,A) \in \mathcal{K}^2_{Top}$ and an abelian
group $G$, there exists a functorial long exact sequence:
\begin{equation}\label{eq1}
	\begin{tikzpicture}\small
		
		\node (A) {$\dots$};
		\node (B) [node distance=2.2cm, right of=A] {$ {\varprojlim} ^ {(3)}   \bar{H}^N _{n+2}(X_\lambda, A_\lambda;G) $};
		\node (C) [node distance=3.2cm, right of=B] {${\varprojlim} ^ {(1)}  \bar{H}^N_{n+1}(X_\lambda,A_\lambda;G)$};
		\node (D) [node distance=2.6cm, right of=C] {$ \bar{H}^N_n(X,A;G)$};
		\node (E) [node distance=2.3cm, right of=D] {${\varprojlim}  \bar{H}^N _{n}(X_\lambda,A_\lambda;G) $};
		\node (F) [node distance=3cm, right of=E] {${\varprojlim} ^ {(2)} \bar{H}^N_ {n+1}(X_\lambda,A_\lambda;G) $};
		\node (H) [node distance=2.2cm, right of=F] {$\dots$~.};
		
		\draw[->] (A) to node [above]{}(B);
		\draw[->] (B) to node [above]{}(C);
		\draw[->] (C) to node [above]{}(D);
		\draw[->] (D) to node [above]{}(E);
		\draw[->] (E) to node [above]{}(F);
		\draw[->] (F) to node [above]{}(H);
		
	\end{tikzpicture}
\end{equation}

$CIG$ ({\it Continuity for an Injective Group}): For each resolution  ${\bf p}:(X,A) \to {\bf (X,A)}=\left\{(X_\lambda, A_\lambda), p_{\lambda \lambda }, \Lambda \right\} $ of closed $P$-pair $(X,A) \in \mathcal{K}^2_{Top}$ and  an injective abelian group $I$, there exists an isomorphism
\begin{equation}\label{eq2}
	\bar{H}^N_n(X,A;I) \approx \varprojlim \bar{H}^N_n(X_\lambda,A_\lambda ;I).
\end{equation}

$UCF$ ({\it Universal Coefficient Formula}): For each $(X,A) \in \mathcal{K}^2_{Top}$ closed $P$-pair and an abelian group $G$, there exists a functorial exact sequence:
\begin{equation}\label{eq3}
	0 \lra \Ext\left(\bar{H}^{n+1}_N(X,A);G\right) \os{}{\lra}  \bar{H}^N_n\left(X,A;G\right) \os{}{\lra} \Hom(\bar{H}^n_N\left(X,A);G\right) \lra 0,
\end{equation}
where $	\bar{H}^{n+1}_N(-,-;G)$ is the Alexander-Spanier normal cohomology \cite{BBM}, \cite{Wat}.

Our aim is to investigate whether the homology theory $\bar{H}^N_{*}(-,-;G)$ is a strong shape invariant {\color{black} or not}. In the paper, we use the method of construction of the strong homology theory to obtain a new strong shape invariant homology theory $\bar{H}^{\infty}_{*}(-,-;G).$ In particular, instead of the singular chain complexes, we use the chain complexes that are constructed in the paper \cite{BM3}, and follow the method of construction of the strong homology theory. {\color{black} Clearly, we obtained that	homology theory} $\bar{H}_{*}^\infty(-,-;G)$  is  a strong shape invariant. Then we prove that the homology theories $\bar{H}_{*}^\infty(-,-;G)$  and $\bar{H}_{*}^N(-,-;G)$  are isomorphic. Therefore, we obtain that the Alexander-Spanier Normal homology theory is a new strong shape invariant exact homology theory.

\section{Homology theory of pro-chain complexes}

In this section, using the chain complexes obtained by the cochain complexes \cite{BM3}, we construct a new strong shape invariant homology $\bar{H}_{*}^\infty(-,-;G)$  groups of a pro-chain complexes. The method of construction is the standard that is considered in \cite{Mar}, \cite{LM1}, \cite{Mdz}, but we recall and review some of them which are important for our purpose.

Let $C^*$ be a cochain complex  and $0 \to G \os{\alpha}{\lra} I^0 \os{\beta}{\lra} I^1 \to 0 $ be an injective resolution of a group $G$. Let $\beta_{\#}:\Hom(C^*;I^0) \to \Hom(C^*;I^1)$ be the chain map induced by $\beta: I^0 \to I^1$. Consider the cone  $C_*(\beta_{\#})=\left\{ C_n(\beta_{\#}), \partial \right\}=\left\{\Hom( C^*,\beta_{\#}), \partial \right\}$ of the chain map $\beta_{\#}$, i.e.
\begin{equation}\label{eq4}
	C_n(\beta_{\#}) \simeq \Hom(C^{n};I^0) \oplus \Hom(C^{n+1};I^1),
\end{equation}
\begin{equation}\label{eq5}
	\partial(\varphi ',\varphi '')=(\varphi ' \circ \delta, \beta \circ \varphi ' -\varphi '' \circ \delta), ~~~ \forall (\varphi ',\varphi '') \in C_n(\beta_{\#}).
\end{equation}
If $f^\#:C^* \lra C'^*$ is a cochain map, then it induces the chain map $f_\#:C_* (\beta'_{\#}) \lra C_*(\beta_{\#})$. In particular, for each $n \in \mathbb{Z}$,  ${f}_n: C_n (\beta'_{\#}) \lra C_n (\beta_{\#})$ is defined by the formula
${f}_n(\varphi ', \varphi '')=(\varphi ' \circ f^n , \varphi '' \circ f^{n+1})$ \cite{BM3}.  

Let ${\bf C}^*=\left\{C^*_\lambda, p_{\lambda \lambda'}, \Lambda\right\}$ be a direct system of cochain complexes. For each abelian group $G$, consider the inverse system ${\bf C}_*=\left\{C^\lambda_*, p_{\lambda \lambda'}, \Lambda\right\}$ of chain complexes $C^\lambda_*=C_*^\lambda(\beta_\#)$. Consider  the total complex $T_*({\bf C}_*)$ \cite{Mar}, i.e.  
\begin{equation}\label{eq6}
	T_n({\bf C}_*)=\prod_{s=0}^{\infty } \prod_{(\lambda_0,\lambda_1, \dots, \lambda_s)}  C^{\lambda_0}_{n+s}=\prod_{s=0}^{\infty } \prod_{{\boldsymbol \lambda} \in \Lambda_s}  C^{\lambda_0}_{n+s}.
\end{equation}
A chain $c = (c_0, c_1, \dots )\in T_n({\bf C}_*)$ is a function which assigns the
element $c({\boldsymbol \lambda})=c_s({\boldsymbol \lambda}) \in C^{\lambda_0}_{n+s}$ to each multiindex ${\boldsymbol \lambda} =(\lambda_0,\lambda_1, \dots \lambda_s) \in \Lambda_s,~s \ge 0.$  The boundary operator $d_n:T_n({\bf C}_*) \to 
T_{n-1}({\bf C}_*)$ is defined by the formula 
\begin{equation}\label{eq7}
	d_n(c)=\partial_n(c)+(-1)^n\delta^n(c),
\end{equation}
 where
\begin{equation}\label{eq8}
	\left( \partial_n(c)\right)({\boldsymbol \lambda})=\partial^{\lambda_0}_{n+s} \left(c({\boldsymbol \lambda})\right),
\end{equation}
 \begin{equation}\label{eq9}
 	\left( \delta_n(c)\right)({\boldsymbol \lambda})=p^{\lambda_0 \lambda_1}_{n+s}c(d^0{\boldsymbol \lambda})\ +\sum_{j=1}^s (-1)^jc(d^j {\boldsymbol \lambda}). 
 \end{equation}
 Note that $d^j {\boldsymbol \lambda} =d^j(\lambda_0,\dots, \lambda_{j-1},\lambda_j,\lambda_{j+1}, \dots , \lambda_s)=(\lambda_0,\dots, \lambda_{j-1},\lambda_{j+1}, \dots , \lambda_s)$, $j=0,1, \dots , s$.
 
 It is known that $T_*\left({\bf C}_*\right)=\left(T_*\left({\bf C}_*\right), d_n\right)$ is a chain complex \cite[Lemma 17.2]{Mar}. Therefore, we define a homology groups of the inverse system ${\bf C}_*$ as the homology groups of the chain complex $T_*\left({\bf C}_*\right)$, i.e.
  \begin{equation}\label{eq10}
 	\bar{H}^{\infty}_*\left({\bf C}_*\right)=H_*\left(T_*\left({\bf C}_*\right)\right).
 \end{equation}
Let ${\bf C}^*=\left\{C^*_\lambda, p_{\lambda \lambda'}, \Lambda\right\}$ and ${\bf D}^*=\left\{D^*_\mu, q_{\mu \mu'}, M \right\}$ be direct systems of cochain complexes and $f^{\#} =(f^{\#}_\lambda, f) : {\bf C}^* \to {\bf D}^*$ be a morphism of direct systems. Consider the corresponding mappings of inverse systems $f_{\#} =(f_{\#}^\lambda, f) : {\bf D}_* \to {\bf C}_*$. Then, according to \cite[\S 17.3]{Mar}, we define the induced chain map  $f_{\#} : T_* \left({\bf D}_*\right) \to  T_* \left({\bf C}_*\right)$  that induces a homomorphism of homology groups \cite[Lemma 17.5]{Mar}:
  \begin{equation}\label{eq10.1}
	f_{*} : \bar{H}^{\infty}_* \left({\bf D}_*\right) \to  \bar{H}^{\infty}_* \left({\bf C}_*\right).
\end{equation}
Consider the truncation $T^{(r)}_*\left({\bf C}_*\right)$ of the complex $T_*\left({\bf C}_*\right)$ at height $r \ge 0$, i.e. 
 \begin{equation}\label{eq11}
  T^{(r)}_n({\bf C}_*)=\prod_{s=0}^{r } \prod_{(\lambda_0,\lambda_1, \dots, \lambda_s)}  C^{\lambda_0}_{n+s}=\prod_{s=0}^{r } \prod_{{\boldsymbol \lambda} \in \Lambda_s}  C^{\lambda_0}_{n+s}.
\end{equation}
The boundary operator $d^{(r)}_n:T^{(r)}_n({\bf C}_*) \to 
T^{(r)}_{n-1}({\bf C}_*)$ {\color{black} is defined in a similar way
	as}   $d_n:T_n({\bf C}_*) \to 
T_{n-1}({\bf C}_*)$. Let 
  \begin{equation}\label{eq12}
 \pi^{(r)}_n :T^{(r)}_n\left({\bf C}_*\right) \to  T^{(r-1)}_{n}\left({\bf C}_*\right), ~~r \ge 1 
 \end{equation}
 be the mapping given by
   \begin{equation}\label{eq13}
 	\pi^{(r)}_n\left(c_0,c_1, \dots, c_r\right)=\left(c_0, c_1, \dots, c_{r-1}\right).
 \end{equation}
Then $\pi^{(r)}_\#:T^{(r)}_*\left({\bf C}_*\right) \to T^{(r-1)}_*\left({\bf C}_*\right)$ is a chain mapping and so, it induces the homomorphism $\pi^{(r)}_n:H_n\left(T^{(r)}_*\left({\bf C}_*\right) \right) \to H_n \left(T^{(r-1)}_*\left({\bf C}_*\right)\right)$ \cite[\S 17.1]{Mar}. For each $r \ge 0$ define the homology group of height $r$ by
   \begin{equation}\label{eq14}
	\bar{H}^{(r)}_n\left({\bf C}_*\right) ={\pi}^{(r+1)}_n\left(H_n\left(T^{(r+1)}_*\left({\bf C}_*\right) \right)    \right) \subset H_n\left(T^{(r)}_*\left({\bf C}_*\right) \right).
\end{equation}
It is known that the homomorphism ${\pi}^{(r)}_n:H_n \left(T^{(r)}_*\left({\bf C}_*\right)\right) \to H_n \left(T^{(r-1)}_*\left({\bf C}_*\right)\right)$ induces the homomorphism
  \begin{equation}\label{eq15}
   	\bar{\pi}^{(r)}_n:\bar{H}^{(r)}_n\left({\bf C}_*\right) \to \bar{H}^{(r-1)}_n\left({\bf C}_*\right).
   \end{equation}
Consequently, we obtain the inverse sequence \cite[\S 17.3]{Mar} of homology groups:
  \begin{equation}\label{eq16}
	\bar{H}^{(*)}_n\left({\bf C}_*\right)= \left(\bar{H}^{(0)}_n\left({\bf C}_*\right)   \os{\bar{\pi}^{(1)}_n}{\leftarrow}     \dots  \os{\bar{\pi}^{(r)}_n}{\leftarrow}  \bar{H}^{(r)}_n\left({\bf C}_*\right)  \os{\bar{\pi}^{(r+1)}_n}{\leftarrow} \bar{H}^{(r+2)}_n\left({\bf C}_*\right)  \os{\bar{\pi}^{(r+2)}_n}{\leftarrow}   \dots   \right).
\end{equation}
{\color{black}The same way can be defined for the natural} projection mappings
  \begin{equation}\label{eq17}
	\pi^{(r)\infty}_\# :T_*\left({\bf C}_*\right) \to  T^{(r)}_{*}\left({\bf C}_*\right), ~~r \ge 0, 
\end{equation}
which induce the mappings
  \begin{equation}\label{eq18}
	\bar{\pi}^{(r)\infty}_n :\bar{H}^\infty_n\left({\bf C}_*\right) \to  \bar{H}^{(r)}_{n}\left({\bf C}_*\right), ~~r \ge 0.
\end{equation}
It is known that the homomorphisms $\bar{\pi}^{(r)\infty}_n$ induce an epimorphism $\bar{\pi}^\infty_n:\bar{H}^\infty\left({\bf C}_*\right)  \os{}{\rightarrow}  \varprojlim \bar{H}^{(r)}_n\left( {\bf C}_*\right)$ and its kernel is ${\varprojlim}^1 H^{(r)}_n\left( {\bf C}_*\right)$, i.e. there is a short exact sequence \cite[Theorem 17.11]{Mar}:
  \begin{equation}\label{eq19}
	0 \os{}{\rightarrow}  {\varprojlim}^{(1)} \bar{H}^{(r)}_{n+1}\left( {\bf C}_*\right) \os{}{\rightarrow}   \bar{H}^{\infty}_n\left( {\bf C}_*\right) \os{\bar{\pi}^{\infty}_n}{\rightarrow} {\varprojlim} \bar{H}^{(r)}_n\left( {\bf C}_*\right) \os{}{\leftarrow} 0.   
\end{equation}

\begin{lemma}\label{lem.1} 
	 For each direct system ${\bf C}^*=\left\{C^*_\lambda, p_{\lambda \lambda'}, \Lambda\right\}$  of cochain complexes and an injective abelian group $I$, the homomorphisms 
	  \begin{equation}\label{eq20}
		\bar{\pi}^{(r)\infty}_n :\bar{H}^\infty_n\left({\bf C}_*\right) \to  \bar{H}^{(r)}_{n}\left({\bf C}_*\right), ~~r \ge 0
	\end{equation}
are isomorphisms, where  ${\bf C}_*=\left\{C^\lambda_*, p_{\lambda \lambda'}, \Lambda\right\}$ is the inverse system of chain complexes $C^\lambda_*=C_*^\lambda \left(\beta_\#\right)$.
\end{lemma}

\begin{proof}  By \cite[Theorem 17.8]{Mar}, for each pro-chain complexes and $n \in \mathbb{Z}$, there exists an exact sequence
  \begin{equation}\label{eq21}
	\dots  \os{}{\rightarrow}  \bar{H}^{(r-2)}_{n+1}\left( {\bf C}_*\right)  \os{}{\rightarrow} {\varprojlim}^{(r)} {H}_{n+r}\left( {\bf C}_*\right) \os{}{\rightarrow}   \bar{H}^{(r)}_n\left( {\bf C}_*\right) \os{\bar{\pi}^{r}_n}{\rightarrow}  \bar{H}^{(r-1)}_n\left( {\bf C}_*\right) \os{}{\rightarrow} {\varprojlim}^{(r+1)} {H}_{n+r}\left( {\bf C}_*\right) \os{}{\rightarrow} \dots~.
\end{equation}	
	Let's show that
  \begin{equation}\label{eq22}
 {\varprojlim}^{(i)} {H}_n \left( {\bf C}_*\right)={\varprojlim}^{(i)} H_n \left( C_*^\lambda\right)=0,~i\ge1.
\end{equation}	
Indeed, by \cite[Theorem 1]{BM3}, for each polyhedron  $\lambda \in \Lambda$ we have the sequence:
\begin{equation}\label{eq23}
	0 \to \Ext\left({H}^{n+1}(C^*_\lambda);G\right) \to \bar{H}_n (C^*_\lambda;G) \to \Hom\left({H}^n(C^*_\lambda);G\right) \to 0,
\end{equation}
where $\bar{H}_n\left(C^*_\lambda;G\right)=H_n\left(C_*^\lambda\left(\beta_\#\right)\right)$. The sequence (\ref{eq23}) induces the long exact sequence: 
\begin{equation}\label{eq24}
	\begin{tikzpicture}
		\node (A) {0};
		\node (B) [node distance=3cm, right of=A] {$\varprojlim \Ext\left({H}^{n+1}(C^*_\lambda),G\right)$};
		\node (C) [node distance=4cm, right of=B] {$\varprojlim \bar{H}_n (C_\lambda^*;G)$};
		\node (D) [node distance=4cm, right of=C] {$\varprojlim \Hom\left({H}^n(C^*_\lambda),G\right)$};
		\node (E) [node distance=3cm, right of=D] {};
		
		\draw[->] (A) to node [above]{}(B);
		\draw[->] (B) to node [above]{}(C);
		\draw[->] (C) to node [above]{}(D);
		\draw[->] (D) to node [above]{}(E);
		
		\node (A1) [node distance=1cm, below of=A] {};
		\node (B1) [node distance=1cm, below of=B] {${\varprojlim}^{(1)} \Ext\left({H}^{n+1}(C^*_\lambda),G\right)$};
		\node (C1) [node distance=1cm, below of=C] {${\varprojlim}^{(1)} \bar{H}_n (C_\lambda^*;G)$};
		\node (D1) [node distance=1cm, below of=D] {${\varprojlim}^{(1)} \Hom\left({H}^n(C^*_\lambda),G\right)$};
		\node (E1) [node distance=1cm, below of=E] {};
		
		\draw[->] (A1) to node [above]{}(B1);
		\draw[->] (B1) to node [above]{}(C1);
		\draw[->] (C1) to node [above]{}(D1);
		\draw[->] (D1) to node [above]{}(E1);
		
		\node (A2) [node distance=1cm, below of=A1] {};
		\node (B2) [node distance=1cm, below of=B1] {${\varprojlim}^{(2)} \Ext\left({H}^{n+1}(C^*_\lambda),G\right)$};
		\node (C2) [node distance=1cm, below of=C1] {${\varprojlim}^{(2)} \bar{H}_n (C_\lambda^*;G)$};
		\node (D2) [node distance=1cm, below of=D1] {${\varprojlim}^{(2)} \Hom\left({H}^n(C^*_\lambda),G\right)$};
		\node (E2) [node distance=1cm, below of=E1] {$\dots ~.$};
		
		\draw[->] (A2) to node [above]{}(B2);
		\draw[->] (B2) to node [above]{}(C2);
		\draw[->] (C2) to node [above]{}(D2);
		\draw[->] (D2) to node [above]{}(E2);
		
	\end{tikzpicture}
\end{equation}
For each injective abelian group $I,$ the functor $\Ext(-;I)$ is trivial and so, there is an isomorphism:
\begin{equation}\label{eq25}
	{\varprojlim}^{(i)} \bar{H}_n (C_\lambda^*;I) \simeq {\varprojlim}^{(i)} \Hom\left(H^n(C^*_\lambda);I\right), ~~~i \ge 0.
\end{equation}
On the other hand, by \cite[Lemma 1.3]{HM}, for each  injective abelian group $I$ we obtain:
\begin{equation}\label{eq27}
	~~ {\varprojlim}^{i} \Hom\left(H^{n}(C^*_\lambda);I\right)=0,~~~i \ge 1.
\end{equation}
Therefore, by (\ref{eq25})  and (\ref{eq27}), we obtain that (\ref{eq22}) is fulfilled. Therefore, in  the sequence (\ref{eq21}), the mappings $\bar{\pi}^{(r)}_n, ~~r\ge 1$  are isomorphisms. Consequently, in the sequence (\ref{eq19}) the group ${\varprojlim}^1 \bar{H}^{(r)}_n\left( {\bf C}_*\right)$ is trivial. Hence, the mapping
  \begin{equation}\label{eq28}
	   \bar{\pi}^{\infty}_n: \bar{H}^{\infty}_n\left( {\bf C}_*\right) \os{}{\rightarrow} {\varprojlim} \bar{H}^{(r)}_n\left( {\bf C}_*\right)   
\end{equation}
is isomorphism. Since, for each injective abelian group $I$ and $r \ge 1$, the mappings $\bar{\pi}^{(r)}_n:\bar{H}^{(r)}_n\left( {\bf C}_*\right) \os{}{\rightarrow} {\varprojlim} \bar{H}^{(r-1)}_n\left( {\bf C}_*\right)$ are isomorphisms, by (\ref{eq19}), the mappings 
  \begin{equation}\label{eq29}
	\bar{\pi}^{(r)\infty}_n:\bar{H}^{\infty}_n\left( {\bf C}_*\right) \os{}{\rightarrow}  \bar{H}^{(r)}_n\left( {\bf C}_*\right),~r\ge 0   
\end{equation}
are isomorphisms as well.
\end{proof}

\begin{corollary}\label{Cor.5}  For each direct system ${\bf C}^*=\left\{C^*_\lambda, p_{\lambda \lambda'}, \Lambda\right\}$   of cochain complexes and an injective abelian group $I$, the mapping 
  \begin{equation}\label{eq30}
	\tilde{\pi}^{(0)\infty}_n :\bar{H}^\infty_n\left({\bf C}_*\right) \to  \varprojlim {H}^{}_{n}\left( C_*^\lambda \right),
\end{equation}
given by 
  \begin{equation}\label{eq31}
	\tilde{\pi}^{(0)\infty}_n \left(\left[z_0,z_1, \dots \right]\right)=\left\{\left[z_0^\lambda\right]\right\}
\end{equation}
is an isomorphism.
\end{corollary}
\begin{proof} Note that the mapping (\ref{eq31}) can be considered as the composition of the following mappings:
\begin{equation}\label{eq32}
  \bar{H}^\infty_n\left({\bf C}_*\right) \os{\bar{\pi}^{(0)\infty}_n}{\rightarrow} \bar{H}^{(0)}_n\left({\bf C}_*\right) \os{{\pi_n}}{\rightarrow} {\varprojlim} {H}_n\left( {C_*^\lambda}\right),
\end{equation}
where the mapping 
\begin{equation}\label{eq33}
	\pi_n: \bar{H}^{(0)}_n\left({\bf C}_*\right) \os{{}}{\rightarrow} {\varprojlim} {H}_n\left( {C_*^\lambda}\right)
\end{equation}	
is given by 
  \begin{equation}\label{eq34}
	\pi_n \left(\left[z_0\right]\right)=\left\{\left[z_0^\lambda\right]\right\}.
\end{equation}
By \cite[Lemma 17.7]{Mar}, the mapping $\pi_n$ is an isomorphism. Hence, by Lemma 1, the mapping (\ref{eq30}) is an isomorphism as well.
\end{proof}

Let ${\bf C}^*=\left\{C^*_\lambda, p_{\lambda \lambda'}, \Lambda\right\}$ be a direct system of cochain complexes and the corresponding limit cochain complex
   \begin{equation}\label{eq35}
 	C^*= \varinjlim C^*_\lambda.
 \end{equation}
Therefore, there is a corresponding mapping ${\bf i}^\#=\left\{i_\lambda ^\#\right\}:\left\{C^*_\lambda , p_{\lambda \lambda '}, \Lambda\right\} \to C^*$, which induces the mapping ${\bf i}_\#=\left\{i^\lambda _\#\right\}:C_* \to \left\{C_*^\lambda , p_{\lambda \lambda '}, \Lambda\right\} $, where $C_*=C_*\left(\beta_\#\right)$ and $C_*^\lambda=C_*^\lambda\left(\beta_\#\right)$, i.e. 
 \begin{equation}\label{eq36}
 	C_*=C_*\left(\beta_\#\right)=Hom\left(C^n;I^0\right)\oplus Hom\left(C^{n+1};I^1\right),
 \end{equation}
 \begin{equation}\label{eq37}
	C^\lambda_*=C^\lambda_*\left(\beta_\#\right)=Hom\left(C^n _\lambda;I^0\right)\oplus Hom\left(C_\lambda^{n+1};I^1\right),
\end{equation}
where $0 \to G \os{\alpha}{\lra} I^0 \os{\beta}{\lra} I^1 \to 0 $ is an injective resolution of $G$.  Consequently, ${\bf i}_\#=\left\{i^\lambda _\#\right\}:C_* \to \left\{C_*^\lambda , p_{\lambda \lambda '}, \Lambda\right\} $ is an inverse limit. Indeed
\[	C_* = Hom\left(C^n ;I^0\right)\oplus Hom\left(C^{n+1};I^1\right)=Hom\left(\varinjlim C^n _\lambda;I^0\right)\oplus Hom\left(\varinjlim  C_\lambda^{n+1};I^1\right)=\]
 \begin{equation}\label{eq38}
	\varprojlim Hom\left( C^n _\lambda;I^0\right)\oplus \varprojlim Hom\left(C_\lambda^{n+1};I^1\right)=\varprojlim C_*^\lambda.
\end{equation}

\begin{lemma}\label{lem.2}  For each direct system ${\bf C}^*=\left\{C^*_\lambda, p_{\lambda \lambda'}, \Lambda\right\}$  of cochain complexes and an abelian group $G$, the mapping
	\begin{equation}\label{eq39}
		h_\#: \varprojlim C_*^\lambda \to T_*\left({\bf C}_*\right)
	\end{equation}
defined by
	\begin{equation}\label{eq40}
	h_n \left(\left\{c^\lambda _0\right\}\right)=h_n \left(c_0\right)=\left(c_0,0, \dots \right), ~\forall c_0=\left\{c_0^\lambda \right\} \in  \varprojlim C_n^\lambda
\end{equation}
is a chain map.
\end{lemma}
\begin{proof} Note that, for each $c_0=\left\{c_0^\lambda\right\} \in \varprojlim C_*^\lambda$, we have $\delta_n (c_0)=0$. Indeed, by (\ref{eq9}), we have
\begin{equation}\label{eq41}
	\delta\left(c_0\right)\left(\lambda_0\lambda_1\right)=p_n^{\lambda_0\lambda_1}\left(c^{\lambda_1}\right)-c^{\lambda_0}=0.
\end{equation}
Therefore, we have
\begin{equation}\label{eq42}
	h_n\left(\partial_n (c_0)  \right)=\left(\partial(c_0),0, \dots \right)=(\partial(c_0),\delta(c_0), 0, \dots)=d_n \left(h_n\left(c_0\right)\right).
\end{equation}
Hence, $h_\#: \varprojlim C_*^\lambda \to T_*\left({\bf C}_*\right)$ is a chain map. 
\end{proof}

\begin{lemma}\label{lem.3}  For each direct system ${\bf C}^*=\left\{C^*_\lambda, p_{\lambda \lambda'}, \Lambda\right\}$  of cochain complexes and an injective  abelian group $I$, the mapping
	\begin{equation}\label{eq43}
		h_*: H_* \left(\varprojlim C_*^\lambda\right)  \to \bar{H}^\infty_* \left({\bf C}_*\right)
	\end{equation}
	induced by
	\begin{equation}\label{eq44}
		h_\#: \varprojlim C_*^\lambda \to T_*\left({\bf C}_*\right),
	\end{equation}
is an isomorphism.
\end{lemma}

\begin{proof}  Let's show that the composition
	\begin{equation}\label{eq45}
		H_* \left(\varprojlim C_*^\lambda\right)  \os{h_*}{\to} \bar{H}^\infty_* \left({\bf C}_*\right)  \os{\tilde{\pi}_*^{(0)\infty}}{\to} \varprojlim H_* \left( C^\lambda_*\right)
	\end{equation}
	is a natural mapping, i.e.
	\begin{equation}\label{eq46}
		\left( \tilde{\pi}_*^{(0)\infty} h_* \right) \left(\left[\left\{z_0^\lambda\right\}\right] \right)=\left\{[z_0^\lambda]\right\}, ~\forall \left[  \left\{ z_0^\lambda \right\} \right] \in H_n \left( \varprojlim C_*^\lambda \right).
	\end{equation}
Indeed, for each $\left[\left\{z_0^\lambda\right\}\right] \in H_n \left( \varprojlim C_*^\lambda \right)$, we have
	 \begin{equation}\label{eq47}
	 	\left( \tilde{\pi}_*^{(0)\infty} h_* \right) \left(\left[\left\{z_0^\lambda\right\}\right] \right)=
	 	\tilde{\pi}_*^{(0)\infty}\left[h_*  \left(\left\{z_0^\lambda \right\} \right)\right]=\tilde{\pi}_*^{(0)\infty}\left[ \left\{z_0^\lambda\right\}, 0, \dots  \right]=\left\{[z_0^\lambda]\right\}. 
	 \end{equation}
On the other hand, by  \cite[Theorem 4]{BM3},  for the direct system ${\bf C}^*=\left\{C^*_\lambda, p_{\lambda \lambda'}, \Lambda\right\}$, there is  a natural exact sequence
	\begin{equation}\label{eq48}
		\begin{tikzpicture}
			
			\node (A) {$\dots$};
			\node (B) [node distance=2cm, right of=A] {$ {\varprojlim} ^ {(3)}   \bar{H}^\lambda_{n+2} $};
			\node (C) [node distance=2.5cm, right of=B] {${\varprojlim} ^ {(1)}  \bar{H}^\lambda_{n+1} $};
			\node (D) [node distance=2.5cm, right of=C] {$ \bar{H}_n \left( {\varinjlim}   C^*_\lambda ;G \right) $};
			\node (E) [node distance=2.5cm, right of=D] {${\varprojlim}  \bar{H}^\lambda_{n} $};
			\node (F) [node distance=2cm, right of=E] {${\varprojlim} ^ {(2)}  \bar{H}^\lambda_{n} $};
			\node (H) [node distance=2cm, right of=F] {$\dots$~,};
			
			\draw[->] (A) to node [above]{}(B);
			\draw[->] (B) to node [above]{}(C);
			\draw[->] (C) to node [above]{}(D);
			\draw[->] (D) to node [above]{}(E);
			\draw[->] (E) to node [above]{}(F);
			\draw[->] (F) to node [above]{}(H);
			
		\end{tikzpicture}
	\end{equation}
	where $\bar{H}_*^\lambda=\bar{H}_* (C^*_\lambda ;G)=H_*(C_*^\lambda).$ Note that by (\ref{eq38}), we have
	\begin{equation}\label{eq49}
		\bar{H}_* (\varinjlim C^*_\lambda ;G) \simeq H_*\left( \varprojlim C_*^\lambda \right) .
	\end{equation}
     Therefore, there is a natural long exact sequence
		\begin{equation}\label{eq50}
		\begin{tikzpicture}
			
			\node (A) {$\dots$};
			\node (B) [node distance=2cm, right of=A] {$ {\varprojlim} ^ {(3)}   {H}_{n+2} \left(C_*^\lambda\right)$};
			\node (C) [node distance=2.8cm, right of=B] {${\varprojlim} ^ {(1)}  {H}_{n+1} \left(C_*^\lambda\right)$};
			\node (D) [node distance=2.5cm, right of=C] {$ {H}_n \left(\varprojlim C_*^\lambda\right) $};
			\node (E) [node distance=2.3cm, right of=D] {${\varprojlim}  {H}_{n}\left(C_*^\lambda\right) $};
			\node (F) [node distance=2.3cm, right of=E] {${\varprojlim} ^ {(2)}  {H}_{n}\left(C_*^\lambda\right) $};
			\node (H) [node distance=2cm, right of=F] {$\dots$~.};
			
			\draw[->] (A) to node [above]{}(B);
			\draw[->] (B) to node [above]{}(C);
			\draw[->] (C) to node [above]{}(D);
			\draw[->] (D) to node [above]{}(E);
			\draw[->] (E) to node [above]{}(F);
			\draw[->] (F) to node [above]{}(H);
			
		\end{tikzpicture}
	\end{equation}
Since, for each injective abelian group $I$, the homology groups ${\varprojlim} ^ {(i)}H_n\left(C_*^\lambda\right)$ are trivial, we obtain that the natural mapping  $	H_* \left(\varprojlim C_*^\lambda\right) \to  \varprojlim H_* \left( C^\lambda_*\right)$, which is {\color{black}the composition of the mappings}
\begin{equation}\label{eq51}
	H_* \left(\varprojlim C_*^\lambda\right)  \os{h_*}{\to} \bar{H}^\infty_* \left({\bf C}_*\right)  \os{\tilde{\pi}_*^{(0)\infty}}{\to} \varprojlim H_* \left( C^\lambda_*\right),
\end{equation}
is an isomorphism. On the other hand, by Lemma \ref{lem.1}, the mapping
  \begin{equation}\label{eq52}
	\tilde{\pi}^{(0)\infty}_n : \bar{H}^\infty_n\left({\bf C}_*\right) \to  \varprojlim H^{}_{n}\left( C_*^\lambda \right)
\end{equation}
is an isomorphism and therefore,
	\begin{equation}\label{eq53}
	h_*: H_* \left(\varprojlim C_*^\lambda\right)  \to \bar{H}^\infty_* \left({\bf C}_*\right)
\end{equation}
is an isomorphism as well.   
\end{proof}

\begin{lemma}\label{lem.4}  Each direct system ${\bf C}^*=\left\{C^*_\lambda, p_{\lambda \lambda'}, \Lambda\right\}$  of cochain complexes and  every short exact sequence of abelian groups
\begin{equation}\label{eq54}
	0 \to G \os{\varphi}{\lra} G_1 \os{\psi}{\lra} G_2 \to 0
\end{equation}
induces a long exact sequence
\begin{equation}\label{eq55}
\dots \os{\psi _*}{\lra} \bar{H}^\infty_{n+1}({\bf C}_*(\beta_{2\#})) \os{d_{n+1}}{\lra} \bar{H}^\infty_n({\bf C}_*(\beta_{\#})) \os{\varphi _*}{\lra} \bar{H}^\infty_n({\bf C}_*(\beta_{1\#})) \os{\psi _*}{\lra} \bar{H}^\infty_n({\bf C}_*(\beta_{2\#})) \os{d_n}{\lra} \dots~~.
\end{equation}
\end{lemma}
\begin{proof} Note that by \cite[Lemma 3]{BM3}, for each cochain complex $C^*_\lambda$, there is a short exact sequence 
\begin{equation}\label{eq56}
	0 \to C^\lambda_*(\beta_{\#}) \os{\varphi}{\lra} C^\lambda_*(\beta_{1 \#}) \os{\psi}{\lra} C^\lambda_*(\beta_{2\#})\to 0.
\end{equation}
Taking direct products of copies of (\ref{eq56}) over all multiindices ${\boldsymbol \lambda} \in \Lambda_s$ and all 
$s \ge 0$, one obtains exact sequences of groups,
\begin{equation}\label{eq57}
	0 \to T_n\left({\bf C}_*(\beta_{\#})\right) \os{\varphi}{\lra} T_n\left({\bf C}_*(\beta_{1\#})\right) \os{\psi}{\lra} T_n\left({\bf C}_*(\beta_{2\#})\right) \to 0.
\end{equation}
 Therefore, there is a short exact sequence of chain complexes
 \begin{equation}\label{eq58}
 	0 \to T_*\left({\bf C}_*(\beta_{2\#})\right) \os{\varphi}{\lra} T_*\left({\bf C}_*(\beta_{1\#})\right) \os{\psi}{\lra} T_*\left({\bf C}_*(\beta_{2\#})\right) \to 0,
 \end{equation}
which induces the sequence (\ref{eq55}).
\end{proof}

\begin{theorem}\label{Th.1}  For each direct system ${\bf C}^*=\left\{C^*_\lambda, p_{\lambda \lambda'}, \Lambda\right\}$  of cochain complexes and an abelian group $G$, the chain map 
\begin{equation}\label{eq59}
	h_\#: \varprojlim C_*^\lambda \to T_*\left({\bf C}_*\right)
\end{equation}
defined by
\begin{equation}\label{eq60}
	h_n \left(\left\{c^\lambda _0\right\}\right)=h_n \left(c_0\right)=\left(c_0,0, \dots \right), ~\forall c_0=\left\{c_0^\lambda \right\} \in  \varprojlim C_n^\lambda,
\end{equation}
induces the isomorphism
\begin{equation}\label{eq61}
	h_*: H_* \left(\varprojlim C_*^\lambda \right) \to \bar{H}_*^\infty \left({\bf C}_*\right).
\end{equation}
\end{theorem}

\begin{proof} Consider the limit cochain complex $C^*=\varinjlim C^*_\lambda.$ Then by the \cite[Lemma 3]{BM3}, for each short exact sequence of abelian groups
	\begin{equation}\label{eq62}
		0 \to G \os{\varphi}{\lra} G_1 \os{\psi}{\lra} G_2 \to 0,
	\end{equation}
	there is  a short exact sequence of chain complexes:
	\begin{equation}\label{eq63}
		0 \to C_*(\beta_{\#}) \os{\varphi}{\lra} C_*(\beta_{1 \#}) \os{\psi}{\lra} C_*(\beta_{2\#})\to 0.
	\end{equation}
Consequently, there is a natural long exact sequence:
\begin{equation}\label{eq64}
	\dots \os{\psi _*}{\lra} \bar{H}_{n+1}(C^*;G_2) \os{d_{n+1}}{\lra} \bar{H}_n({ C}^*;G) \os{\varphi _*}{\lra} \bar{H}_n({C}^*;G_1) \os{\psi _*}{\lra} \bar{H}_n({ C}^*;G_2) \os{d_n}{\lra} \dots~~.
\end{equation}	
By the 	(\ref{eq38}), $\bar{H}_n({C}^*;G)=H_n(C_*)=H_n(\varprojlim C_*^\lambda(\beta_\#))$ and so, the sequence (\ref{eq64}) can be re-written in the form:
\begin{equation}\label{eq65}
	\dots \os{\psi _*}{\lra} {H}_{n+1}\left(\varprojlim  C^\lambda_*\left(\beta_{2\#}\right)\right) \os{d_{n+1}}{\lra} {H}_n \left(\varprojlim  C^\lambda_*\left(\beta_{\#}\right)\right) \os{\varphi _*}{\lra} {H}_n \left(\varprojlim  C^\lambda_*\left(\beta_{1\#}\right)\right) \os{\psi _*}{\lra} {H}_n \left(\varprojlim  C^\lambda_*\left(\beta_{2\#}\right)\right) \os{d_n}{\lra} \dots~~.
\end{equation}		
By the 	Lemma \ref{lem.4}, we have the exact sequence (\ref{eq55}) as well. Consequently, by the Lemma \ref{lem.2}, the mapping (\ref{eq39}) induces the following commutative diagram:

\begin{equation}\label{eq66}
	\begin{tikzpicture}
		
		\node (A) {$\dots$};
		\node (B) [node distance=2cm, right of=A] {$ {H}_{n+1}\left(\varprojlim  C^\lambda_*\left(\beta_{1\#}\right)\right) $};
		\node (C) [node distance=3.2cm, right of=B] {$ {H}_{n+1}\left(\varprojlim  C^\lambda_*\left(\beta_{2\#}\right)\right)$};
		\node (D) [node distance=3.2cm, right of=C] {$ {H}_{n}\left(\varprojlim  C^\lambda_*\left(\beta_{\#}\right)\right) $};
		\node (E) [node distance=3.2cm, right of=D] {$ {H}_{n}\left(\varprojlim  C^\lambda_*\left(\beta_{1\#}\right)\right) $};
		\node (F) [node distance=3.2cm, right of=E] {$ {H}_{n}\left(\varprojlim  C^\lambda_*\left(\beta_{2\#}\right)\right) $};
		\node (H) [node distance=2cm, right of=F] {$\dots$~};
		
		\draw[->] (A) to node [above]{}(B);
		\draw[->] (B) to node [above]{}(C);
		\draw[->] (C) to node [above]{}(D);
		\draw[->] (D) to node [above]{}(E);
		\draw[->] (E) to node [above]{}(F);
		\draw[->] (F) to node [above]{}(H);

		\node (A) {$\dots$};
		\node (A1) [node distance=1.5cm, below of=A] {$\dots$};
		\node (B1) [node distance=1.5cm, below of=B] {$ \bar{H}^\infty_{n+1}({\bf C}_*(\beta_{1\#})) $};
		\node (C1) [node distance=1.5cm, below of=C] {$ \bar{H}^\infty_{n+1}({\bf C}_*(\beta_{2\#})) $};
		\node (D1) [node distance=1.5cm, below of=D] {$ \bar{H}^\infty_{n}({\bf C}_*(\beta_{\#})) $};
		\node (E1) [node distance=1.5cm, below of=E] {$ \bar{H}^\infty_{n}({\bf C}_*(\beta_{1\#})) $};
		\node (F1) [node distance=1.5cm, below of=F] {$ \bar{H}^\infty_{n}({\bf C}_*(\beta_{2\#})) $};
		\node (H1) [node distance=1.5cm, below of=H] {$\dots$~.};
		
		\draw[->] (A1) to node [above]{}(B1);
		\draw[->] (B1) to node [above]{}(C1);
		\draw[->] (C1) to node [above]{}(D1);
		\draw[->] (D1) to node [above]{}(E1);
		\draw[->] (E1) to node [above]{}(F1);
		\draw[->] (F1) to node [above]{}(H1);
		
		\draw[->] (B) to node [right]{$h_{1*}$}(B1);
		\draw[->] (C) to node [right]{$h_{2*}$}(C1);
		\draw[->] (D) to node [right]{$h_*$}(D1);
		\draw[->] (E) to node [right]{$h_{1*}$}(E1);
		\draw[->] (F) to node [right]{$h_{2*}$}(F1);
		
	\end{tikzpicture}
\end{equation}
On the other hand, if the sequence (\ref{eq62}) is an injective resolution of the group $G$, then by Lemma \ref{lem.3}, the homomorphisms 
\begin{equation}\label{eq67}
      	h_{1*}: {H}_{*}\left(\varprojlim  C^\lambda_*\left(\beta_{1\#}\right)\right)  \to \bar{H}^\infty_* \left({\bf C}_*(\beta_{1\#})\right)
\end{equation}	
and 
\begin{equation}\label{eq68}
	h_{2*}: {H}_{*}\left(\varprojlim  C^\lambda_*\left(\beta_{2\#}\right)\right)  \to \bar{H}^\infty_* \left({\bf C}_*(\beta_{2\#})\right)
\end{equation}	
are isomorphisms. Therefore, by the diagram (\ref{eq66}),
\begin{equation}\label{eq68}
	h_*: {H}_{*}\left(\varprojlim  C^\lambda_*\left(\beta_{\#}\right)\right)  \to \bar{H}^\infty_* \left({\bf C}_*(\beta_{\#})\right)
\end{equation}
is isomorphism as well.	
\end{proof}

\section{A new strong shape invariant homology theory}

In this section, according to \cite[\S 19.1]{Mar}, we define homology groups $ \bar{H}^\infty_* \left(X;G\right)$ of topological spaces, which are strong shape invariant, and find the connection with Alexsader Spanier normal cohomology groups. In particular, let ${\bf p}:X \to {\bf X}=\left\{X_\lambda, p_{\lambda \lambda ' }, \Lambda\right\} $  be a polyhedral resolution. Consider the direct system $\mathbf{C}^*_N=\left\{\bar{C}^*_N(X_\lambda; \mathbb{Z} ), p_{\lambda \lambda ' }, \Lambda\right\}$ of the cochain complexes, where  $\bar{C}^*_N(X_\alpha;\mathbb{Z})$ is the Alexander-Spanier cochain complex based on the normal coverings \cite{BBM}. Then, for each abelian group $G$, we have the corresponding inverse system $\mathbf{C}_*^N=\left\{\bar{C}_*^N(X_\lambda; G ), p_{\lambda \lambda ' }, \Lambda\right\}$  of chain complexes. Note that a chain complex  $\bar{C}_*^N(X_\lambda; G )$ is defined using the injective resolution  $0 \to G \os{\alpha}{\lra} I^0 \os{\beta}{\lra} I^1 \to 0 $ of $G$ as in \cite{BBM}. For simplicity, instead of $\mathbf{C}^*_N=\left\{\bar{C}^*_N(X_\lambda; \mathbb{Z} ), p_{\lambda \lambda ' }, \Lambda\right\}$  and  $\mathbf{C}_*^N=\left\{\bar{C}_*^N(X_\lambda; G ), p_{\lambda \lambda ' }, \Lambda\right\}$, we use the notations $\mathbf{C}^*_N=\left\{{C}^*_\lambda, p_{\lambda \lambda ' }, \Lambda\right\}$  and  $\mathbf{C}_*^N=\left\{{C}_*^\lambda, p_{\lambda \lambda ' }, \Lambda\right\}$, respectively.

Hence, for each topological space $X$ and an abelian group $G$, we have an inverse system  $\mathbf{C}_*^N=\left\{{C}_*^\lambda, p_{\lambda \lambda ' }, \Lambda\right\}$ of chain complexes. Consider the corresponding total complex $T_*({\bf C}_*^N)$. Then, the homology groups of the complex  $T_*({\bf C}_*^N)$ are considered as a homology groups of a topological space $X$, i.e
\begin{equation}\label{eq71}
	 \bar{H}^\infty_* \left(X;G\right)=H_*\left(T_*({\bf C}_*^N)\right).
\end{equation}
Note that the construction of the groups $\bar{H}^\infty_* \left(X;G\right)$ is absolutely the same as the construction of the strong homology groups of $X$ as in \cite{Mar}. The only deference is that instead of singular chain complexs $S_*(X_\lambda)$ of a space $X_\lambda$, we have considered the Alexsander Spanier chain complexes $\bar{C}_*^N(X_\lambda; G )$ defined in \cite{BBM}. Consequently, in the same way, we define the induced homomorphism for each strong shape morphism $F:X \to Y$ and for each continuous map $f: X \to Y$. As a result, we obtain that the constructed homology functor is strong shape invariant (cf. \cite[Corollary 19.2]{Mar}):

\begin{corollary}\label{Cor.2} The function which assigns the homology groups $\bar{H}^\infty_* \left(X;G\right)$ and the induced homomorphism $f_* : \bar{H}^\infty_* \left(X;G\right) \to \bar{H}^\infty_* \left(Y;G\right)$ to a space $X$ and a mapping $f:X \to Y$, respectively, is a strong shape invariant functor
	\begin{equation}\label{eq72}
		\bar{H}^\infty_* \left(-;G\right):  \mathcal{K}_{Top} \to  \mathcal{K}_{Ab}.
	\end{equation}
Therefore, if the mappings $f,g :X \to Y$ induce the same strong shape morphism, then
	\begin{equation}\label{eq73}
	f_*=g_* : \bar{H}^\infty_* \left(X;G\right) \to \bar{H}^\infty_* \left(Y;G\right).
\end{equation}
\end{corollary}

\begin{theorem}\label{Th.2}  	For each topological space $X$ and abelian group $G$ there is an isomorphism
	\begin{equation}\label{eq74}
		\bar{H}_*^N\left(X;G\right) \simeq \bar{H}^{\infty}_*\left(X;G\right).
	\end{equation}
\end{theorem}

\begin{proof} By the \cite[Theorem 4]{BBM}, for each topologycal space $X$, there is a cochain isomorphism
	\begin{equation}\label{eq75}
		{\bf p}^\#: \varinjlim \bar{C}^*_N(X_\lambda) \to \bar{C}^*_N(X),
	\end{equation}
induced by a polyhedron resolution ${\bf p}:X \to {\bf X}=\left\{X_\lambda, p_{\lambda \lambda ' }, \Lambda \right\} $.
Therefore, by (\ref{eq38}), we have 
	\begin{equation}\label{eq75}
	\bar{H}_*^N(X;G) \simeq   \bar{H}_* \left( \varinjlim \bar{C}^*_N(X_\lambda); G \right) \simeq H_*\left( \varprojlim \bar{C}_*^N(X_\lambda; G) \right) .
\end{equation}
On the other hand, by Lemma \ref{lem.3}, we have
\begin{equation}\label{eq76}
	 H_*\left( \varprojlim \bar{C}_*^N(X_\lambda; G) \right)  \simeq H^\infty_* \left({\bf C}_*^N\right).
\end{equation}
Therefore, we obtain that
\begin{equation}\label{eq77}
	\bar{H}_*^N(X;G) \simeq H_*\left( \varprojlim \bar{C}_*^N(X_\lambda; G) \right) \simeq H^\infty_* \left({\bf C}_*^N\right)  \simeq H_* \left(T_* \left({\bf C}_*^N\right) \right) \simeq \bar{H}^{\infty}_*\left(X;G\right).
\end{equation}
\end{proof}

\begin{corollary}\label{Cor.3} The Alexander-Spannier Normal Homology Theory $\bar{H}_*^N\left(-, -;G\right)$ is a strong shape invariant.
\end{corollary}

{\color{black} Note that by \cite[Corollary 9]{BBM}, the Alexander-Spanier normal} homology theory $\bar{H}_*^N(-,-;G)$, defined on the category $\mathcal{K}^2_{Top}$, satisfies $UCF$ property ({\it Universal Coefficient Formula}): for each $(X,A) \in \mathcal{K}^2_{Top}$ closed $P$-pair and an abelian group $G$, there exists a functorial exact sequence:
\begin{equation}\label{eq46}
	0 \lra \Ext\left(\bar{H}^{n+1}_N(X,A);G\right) \os{}{\lra}  \bar{H}_n^N\left(X,A;G\right) \os{}{\lra} \Hom(\bar{H}^n_N\left(X,A);G\right) \lra 0,
\end{equation}
where $	\bar{H}^{n+1}_N(-,-;G)$ is the Alexander-Spanier normal cohomology. On the other hand, by \cite[Theorem 4]{MP2} the Universal Coefficient Formula is valid for the strong homology theory $\bar{H}_*\left(-,-;G\right)$ on the category of pairs of compact topological spaces $\mathcal{K}^2_{C}$. Therefore, by the uniqueness theorem given by Berikashvili \cite{Be}, \cite{BM2}, the Alexander-Spanier normal homology theory $\bar{H}^N_{*}(-,-;G)$ \cite{BBM} is isomorphic to the strong homology theory $\bar{H}_{*}(-,-;G)$ on the category $\mathcal{K}^2_C$. On the other hand, for the strong homology theory  $\bar{H}_{*}(-,-;G)$ the Universal Coefficient Formula is false even for polyhedrons \cite[Remark 6]{MP2}. Therefore, in general, the homology theories $\bar{H}^N_{*}(-,-;G)$  and $\bar{H}_{*}(-,-;G)$  are different. 

\begin{corollary}\label{Cor.4} 		The Alexander-Spannier Normal Homology Theory $\bar{H}_*^N\left(-, -;G\right)$ is isomorphic to the strong homology theory $\bar{H}_*(- ,- ;G)$ on the category of pairs of compact topological spaces $\mathcal{K}^2_{C}$, but not in general. Therefore, there exist different exact strong invariant homology theories $\bar{H}_*^N\left(-, -;G\right)$ and $\bar{H}_*(- ,- ;G)$ on the category of pairs of topological spaces $\mathcal{K}^2_{Top}$.
\end{corollary}

\section*{Acknowledgements}

The work was supported by Shota Rustaveli National Science Foundation of Georgia (SRNSF grant FR-23-271).\\

We extend our sincere thanks to the reviewers for their valuable comments and suggestions, which have enhanced the exploration of our results.


\bibliographystyle{elsarticle-num}

\end{document}